\documentclass[11 pt, leqno]{amsart}

\usepackage {amsfonts}
\usepackage{amsthm}
\usepackage{amssymb}
\usepackage{latexsym}
\usepackage{amsmath}
\usepackage{mathrsfs}
\usepackage{a4wide}
\usepackage{comment}

\theoremstyle{plain}
\newtheorem{tw}{Theorem}[section]

\newtheorem {lem} [tw]{Lemma}
\newtheorem {prop}[tw] {Proposition}

\newtheorem{cor}[tw]{Corollary}
\newtheorem{quest}{Question}[section]

\theoremstyle{definition}
\newtheorem {deft}[tw] {Definition}
\newtheorem {rem} [tw]{Remark}

\newcommand{\bc} {\Bbb C}
\newcommand{\bn}{\Bbb N}
\newcommand{\br}{\Bbb R}

\newcommand{\QGam}{\Gamma}

\newcommand {\Pol} {{\textup{Pol}}}
\newcommand {\Alg} {\mathcal{A}}
\newcommand {\id} {{\textrm{id}}}

\newcommand{\tu}{\textup}

\newcommand{\Hil}{\mathsf{H}}

\newcommand{\Gil}{\mathsf{G}}
\newcommand{\Ril}{\mathsf{R}}

\newcommand{\Irr}{\textup{Irr}}
\newcommand{\ol}{\overline}

\newcommand{\Ker}{\tu{Ker}}
\newcommand{\QG}{\mathbb{G}}

\newcommand{\mean}{\mathcal{M}}
\newcommand{\Com}{\Delta}
\newcommand{\Cou}{\epsilon}

\newenvironment{rlist}
{

\begin{enumerate}}
{\end{enumerate}}

\newcommand{\la}{\langle}
\newcommand{\ra}{\rangle}
\newcommand{\PolQG}{\Pol(\QG)}

\newcommand{\ot}{\otimes}

\newcommand{\wt}{\widetilde}

\numberwithin{equation}{section}

\begin{document}


\title[Generating functionals, cocycles and symmetry]{One-to-one correspondence between generating functionals and cocycles on quantum groups in presence of symmetry}

\begin{abstract}
We prove that under a symmetry assumption all cocycles on Hopf $^*$-algebras arise from generating functionals. This extends earlier results of R.Vergnioux  and D. Kyed and has two quantum group applications: all quantum L\'evy processes with symmetric generating functionals decompose into a maximal Gaussian and purely non-Gaussian part and the Haagerup property for discrete quantum groups is characterized by the existence of an arbitrary proper cocycle.
\end{abstract}


\author{Biswarup Das}
\address{Institute of Mathematics of the Polish Academy of Sciences,
ul.~\'Sniadeckich 8, 00--656 Warszawa, Poland}
\email{B.Das@impan.pl}
\thanks{BD acknowledges the support of  WCMS postdoctoral grant}

\author{Uwe Franz}
\address{D\'epartement de math\'ematiques de Besan\c{c}on,
Universit\'e de Franche-Comt\'e 16, route de Gray, 25 030
Besan\c{c}on cedex, France}
\email{uwe.franz@univ-fcomte.fr}
\urladdr{http://www-math.univ-fcomte.fr/pp\underline{ }Annu/UFRANZ/}
\thanks{UF was supported by an ANR Project OSQPI (ANR-11-BS01-0008).}

\author{Anna Kula}
\address{Instytut Matematyczny, Uniwersytet Wroc\l awski, pl.Grunwaldzki 2/4, 50-384 Wroc\l aw, Poland
\newline \indent
Instytut Matematyki, Uniwersytet Jagiello\'nski, ul.\L ojasiewicza 6, 30--348 Krak\'ow, Poland}
\email{Anna.Kula@math.uni.wroc.pl}
\thanks{AK was partially supported by the PostDoctoral Fellowship  2012/04/S/ST1/00102 of the Polish National Science Center.}

\author{Adam Skalski}
\address{Institute of Mathematics of the Polish Academy of Sciences,
ul.~\'Sniadeckich 8, 00--656 Warszawa, Poland \newline \indent Faculty of Mathematics, Informatics and Mechanics, University of Warsaw, ul.~Banacha 2,
02-097 Warsaw, Poland
\newline \indent CNRS, D\'epartement de math\'ematiques de Besan\c{c}on,
Universit\'e de Franche-Comt\'e 16, route de Gray, 25 030
Besan\c{c}on cedex, France}
\thanks{AS is partially supported by the HARMONIA NCN grant
2012/06/M/ST1/00169}
\email{a.skalski@impan.pl}


\maketitle

The correspondence between (real) conditionally positive definite functions on a discrete group $\Gamma$ and cocycles for orthogonal representations of $\Gamma$ on real Hilbert spaces plays a key role in the functional-analytic geometric group theory, allowing a transition between analytic and geometric data (see for example the dictionary in the beginning of Chapter 2 in \cite{bdv}). Perhaps less well-known, but equally important, is its role in the algebraic approach to L\'evy processes on groups and L\'evy-Khintchine decomposition. When we pass to the quantum context, replacing classical groups by their quantum counterparts, the relation between the two corresponding notions, i.e.\ \emph{generating functionals} and \emph{cocycles} on $\PolQG$, where $\QG$ is a compact quantum group,  becomes somewhat more involved, but its importance does not diminish. This was first recognized by M.Sch\"urmann (\cite{SchuProc}, \cite{Schurmann}), who in particular used the GNS construction leading from a generating functional to a cocycle as a key step in the reconstruction theorem for quantum L\'evy processes. He also noted the connection between the possibility of attributing a generating functional to a given cocycle and the extraction of a maximal Gaussian part of a given quantum L\'evy process. It was soon realized that in general given a cocycle the corresponding generating functional need not exist (\cite{SchSk}, \cite{Skeide}).  Almost twenty years later D.Kyed, basing his result on the unpublished notes by R.Vergnioux, showed that \emph{real} cocycles always arise from generating functionals. Reality is now understood as a specific interaction with the antipode $S$ of $\PolQG$. In Kyed's paper this was used to study Kazhdan's Property (T) for discrete quantum groups; recently the same construction was employed to the analysis of the Haagerup property in \cite{DawFimSkaWhi}.

In this work we extend the results of Vergnioux and Kyed to the situation where the reality condition is in addition twisted by what we call an \emph{admissible bijection}. A primary motivating example arises from a scaling action on the algebra of functions on a non-Kac compact quantum group, where the new reality condition can be viewed as related to the \emph{unitary antipode} of $\PolQG$. We show that in fact for any Hopf $^*$-algebra $\Alg$ and an admissible bijection $\alpha:\Alg \to \Alg$ there is a one-to-one correspondence between \emph{$\alpha$-real} cocycles on $\Alg$ and $S \circ \alpha$-invariant generating functionals on $\Alg$. The difficult part is the one where we associate a functional to a given cocycle, the other direction follows from Sch\"urmann's GNS construction. It deserves to be noted that the construction, although purely algebraic, is more involved than the one of Vergnioux-Kyed. The result turns out to have some important applications: it enables us to show that any quantum L\'evy process whose generating functional is $S\circ \alpha$-invariant allows the extraction of its maximal Gaussian part and that the Haagerup property for a discrete quantum group is characterised by the existence of a proper (not necessarily real) cocycle. The latter result strengthens Theorem 7.23 in \cite{DawFimSkaWhi}.

The plan of the article is as follows: in Section 1 we describe the notation and terminology, in particular developing the concept of the admissible bijection, quoting known results regarding the correspondence between cocycles and generating functionals and presenting a homological viewpoint on the problem studied in the paper. In Section 2 we prove the main result and in Section 3 establish two applications mentioned above.

  All the inner products in the article are linear on the right.

\section{Notations and preliminaries} \label{NotPrel}

Let $\Alg$ be a Hopf $^*$-algebra (\cite{KlimykSchmudgen}). We will denote its coproduct by $\Com$, the counit by $\Cou$ and the antipode by $S$, and very often employ \emph{Sweedler's notation}: if $a \in \Alg$ then
\[ \Com(a) = a_{(1)} \ot  a_{(2)}.\]
This requires certain care: so for example, as the interaction of the coproduct with the antipode involves a tensor flip $\Sigma:\Alg \ot \Alg \to \Alg \ot \Alg$, we have $(Sa)_{(1)} = S( a_{(2)})$, $(Sa)_{(2)} = S( a_{(1)})$. The defining antipode relation in Sweedler's notation takes the following form:
\begin{equation} S(a_{(1)}) a_{(2)} = \Cou(a) 1 = a_{(1)} S(a_{(2)}), \;\; a \in \Alg.
\label{antipode}\end{equation}
Occasionally we will also need a triple version of the Sweedler notation: $\Com^{(2)}(a) = a_{(1)} \ot  a_{(2)} \ot  a_{(3)}$, where $a \in \Alg$ and $\Com^{(2)}:= (\id \ot \Com)\circ \Com = ( \Com \ot \id)\circ \Com$. The adjoint map of $\Alg$ will be denoted simply by $^*$, so that we have for example a relation $S \circ ^* S\circ ^* = \id$. We will further denote the ideal $\Ker (\Cou)$  by $K_1$.

The main motivating examples  we  have in mind are those of Hopf $^*$-algebras coming from compact quantum groups. Let then $\QG$ be a compact quantum group in the sense of Woronowicz (\cite{wor1}, \cite{wor2}) -- note it is defined implicitly, in terms of its \emph{algebra of continuous functions}, a unital $C^*$-algebra $\textup{C}(\QG)$, equipped with a coproduct $\Com: \textup{C}(\QG) \to \textup{C}(\QG) \ot^{\textup{sp}} \textup{C}(\QG)$, where $\ot^{\tu{sp}}$ denotes the spatial tensor product of $C^*$-algebras. A unitary matrix $U=(u_{ij})_{i,j=1}^n \in M_n (\textup{C}(\QG))$ is called a (finite-dimensional) \emph{unitary representation} of $\QG$ if $\Com(u_{ij})= \sum_{k=1}^n u_{ik} \ot u_{kj}$, $i,j=1,\ldots,n$; each of the elements $u_{ij}$ is called a \emph{coefficient} of $U$. The linear span of all coefficients of finite dimensional unitary representations of $\QG$ is a dense unital $^*$-subalgebra of $\textup{C}(\QG)$, which turns out to have the structure of a Hopf $^*$-algebra with the coproduct inherited from $\textup{C}(\QG)$. There are natural notions of irreducibility and unitary equivalence for unitary representations of $\QG$; if we denote by $\Irr (\QG)$ the set of all equivalence classes of irreducible representations of $\QG$ and for each $\beta \in \Irr (\QG)$ choose a representative $U^{\beta} \in M_{n_{\beta}}(\Pol(\QG))$ then $\{u^{\beta}_{ij}:\beta \in \Irr (\QG), i,j=1, \ldots, n_{\beta}\}$ forms a linear basis of $\Pol(\QG)$. The Hopf $^*$-algebra $\Pol(\QG)$ admits a \emph{scaling automorphism group}, i.e.\ a one-parameter group of automorphisms $(\tau_t)_{t \in \br}$ which is `locally implemented', i.e.\ for each $\beta \in \Irr(\QG)$ there exists a positive invertible matrix $Q^{\beta} \in M_{n_{\beta}}$ such that we have
\[(\tau_t \ot \id_{M_{n_{\beta}}} )(U^{\beta}) = (Q^{\beta})^{it} U^{\beta} (Q^{\beta})^{-it}.\]
The above formula (or the implementation in terms of so-called \emph{Woronowicz characters}) implies that we can in fact replace $t$ above by any  number $z\in \bc$ and still obtain a bijective homomorphism $\tau_z:\Pol(\QG)\to \Pol(\QG)$, such that $\tau_z(a^*) = (\tau_{\bar{z}}(a))^*$ for all $a \in \Pol(\QG)$. Each $\tau_z$ commutes with the antipode and also intertwines the coproduct: $(\tau_z \ot \tau_z) \circ \Com = \Com \circ \tau_z$; moreover $\tau_z \circ \tau_w = \tau_{z+w}$ for all $z, w \in \bc$.

We can assume, at the cost of possibly choosing another representative of $\beta$, that the matrix $Q^{\beta}$ is diagonal: this means that there exist strictly positive numbers $q_1(\beta), \ldots, q_{n_{\beta}}(\beta)$ such that for each $z\in \bc$ and $i,j=1,\ldots, n_{\beta}$ we have
\[ \tau_z(u_{ij}^{\beta}) = \left( \frac{q_i(\beta)}{q_j(\beta)} \right)^{iz} u_{ij}^{\beta}.\]
The \emph{unitary antipode} $R$ of $\QG$ is a $^*$-preserving involutive anti-automorphism of $\Pol(\QG)$ defined as $R=S\circ \tau_{\frac{i}{2}}$; in the basis chosen above one obtains thus
\[ R(u_{ij}^{\beta}) = \left( \frac{q_j(\beta)}{q_i(\beta)} \right)^{\frac{1}{2}} (u_{ji}^{\beta})^*.\]
All these facts can be located in \cite{wor2}, see also \cite{KlimykSchmudgen} (note that we follow rather a convention used for example in \cite{Johan} or \cite{DawFimSkaWhi} then this of \cite{wor2} -- in the latter article one has $R=S\circ \tau_{-\frac{i}{2}}$). Finally observe that the scaling automorphisms are non-trivial if and only if $\QG$ is not of the \emph{Kac type}, that is if the Haar state of $\QG$ is not tracial. In particular they trivialise when we study classical compact groups $G$.

Finally let us recall that the Hopf $^*$-algebras arising as $\Pol(\QG)$ for a certain compact quantum group $\QG$ have intrinsic characterization, as \emph{CQG algebras} (\cite{DiK}).

\subsection*{Admissible bijections}
The following definition plays a crucial role in this article.

\begin{deft}\label{Def:admissible}
Let $\Alg$ be a Hopf$^*$-algebra. A map $\alpha:\Alg \to \Alg$ is called an admissible bijection if it satisfies the following conditions:
\begin{rlist}
\item $\alpha$ is a homomorphism;
\item $\alpha \circ ^* \circ \alpha \circ ^* = \id$;
\item $(\alpha \ot \alpha)\circ \Com = \Com \circ \alpha$;
\item the linear map $(\id + \alpha): \Alg \to \Alg$ is a bijection;
\item the linear map $(\id \ot \id + \alpha \ot \alpha): \Alg \ot \Alg\to \Alg \ot \Alg$ is a bijection.
\end{rlist}
\end{deft}

It is not difficult to see that if an admissible bijection is $^*$-preserving (in other words it is an automorphism of the Hopf $^*$-algebra $\Alg$), it must be an identity map -- indeed, if $\alpha$ is $^*$-preserving then (ii) above implies that it is an order two automorphism, so $\Alg$ decomposes into  eigenspaces of $\alpha$ corresponding respectively to eigenvalues 1 and -1. If the second of these is non-trivial, then condition (iv) cannot hold.

Given an admissible bijection $\alpha:\Alg \to \Alg$ we define an \emph{$\alpha$-twisted antipode} $S_{\alpha}$ via the formula
\[S_{\alpha} =  S \circ \alpha.\]

\begin{prop} \label{tauadmissible}
Let $\QG$ be a compact quantum group, $t\in \br$. Then the map $\tau_{it}: \Pol(\QG) \to \Pol(\QG)$ is an admissible bijection. In particular the unitary antipode $R=S \circ \tau_{\frac{i}{2}}$ is an $\alpha$-twisted antipode for $\alpha = \tau_{\frac{i}{2}}$.
\end{prop}
\begin{proof}
Let $t \in \br$. The discussion in the beginning of this section shows that the first three conditions are satisfied; for the second remark just that we have
\[\tau_{it} \circ ^* \circ \tau_{it} \circ ^* = ^* \circ \tau_{\overline{it}}  \circ \tau_{it} \circ ^* = \id.\]
Further note that in the basis $\{u^{\beta}_{ij}:\beta \in \Irr (\QG), i,j=1, \ldots, n_{\beta}\}$ of $\Pol(\QG)$ the linear map $\id + \tau_{it}$ is a diagonal operator with respective eigenvalues $1+ \left(\frac{q_i(\beta)}{q_j(\beta)}\right)^{-t}  \neq 0$. Similarly in the basis $\{u^{\beta}_{ij} \ot u^{\beta'}_{kl} :\beta, \beta' \in \Irr (\QG), i,j=1, \ldots, n_{\beta}, k,l=1, \ldots, n_{\beta'}\}$ of $\Pol(\QG) \ot \Pol(\QG)$ the map $\id \ot \id + \tau_{it} \ot \tau_{it}$ is diagonal with non-zero eigenvalues $1+ \left(\frac{q_i(\beta)q_k(\beta')}{q_j(\beta)q_l(\beta')}\right)^{-t} \neq 0$. This shows that (iv)-(v) of Definition \eqref{Def:admissible} also hold.
\end{proof}

\begin{rem}
The comment after Definition \ref{Def:admissible} shows that to find non-trivial admissible bijections one needs to deform the $^*$-structure of $\Alg$. The construction in Proposition \ref{tauadmissible}, of course providing new examples only in the non-Kac case, suggests the following example. Let $N \in \bn$, $N\geq 2$ and consider the unitary group $U(N)$ with $(U_{ij})_{i,j=1}^N$ denoting the standard generators of $\Pol(U(N))$. Further consider a tuple $(q_1,\ldots,q_N)$ of strictly positive real numbers. An identification of $\Alg:=\Pol(U(N))$ as the universal commutative algebra generated by $2N^2$ variables $\{U_{ij}, U_{kl}^*:i,j,k,l=1, \ldots N\}$ satisfying the unitarity relations $\sum_{k=1}^n U_{ik} U_{jk}^* = \sum_{k=1}^n U_{ki} U_{kj}^*=\delta_{ij} I$ for each $i,j=1,\ldots,n$ allows us to define a bijective homomorphism $\alpha$ on $\Alg$ via the homomorphic extension of the formulas
\[ \alpha(U_{ij}) = q_i q_j^{-1} U_{ij}, \;\;\; \alpha(U_{kl}^*) = q_l q_k^{-1} U_{kl}^*, \;\;\; i,j,k,l=1, \ldots N.\]
An explicit, long calculation shows that $\alpha$ is indeed an admissible bijection, different from the identity map if  not all of the numbers $q_i$ coincide.

\end{rem}

We will now collect some basic properties of admissible bijections and related twisted antipodes.

\begin{prop} \label{propadmissbij}
Let $\alpha:\Alg \to \Alg$ be an admissible bijection. Then
\begin{rlist}
\item $\alpha$ is a bijection, $\alpha^{-1} = ^* \circ \alpha \circ ^*$ is an admissible bijection;
\item $\alpha (1) =1$ (so also $S_{\alpha}(1) = 1$);
\item $\Cou \circ \alpha =\Cou$ (so also $\Cou \circ S_{\alpha} = \Cou$);
\item $\alpha \circ S = S \circ \alpha$ (so that $S_{\alpha} = \alpha \circ S $);
\item $S_\alpha \circ ^* \circ S_\alpha \circ ^* = \id$ and $S_{\alpha}$ is an anti-homomorphism;
\item $(S_\alpha \ot S_\alpha) \circ \Sigma \circ \Com  = \Com \circ S_{\alpha}$;
\item the twisted antipode relation holds:
\begin{equation} S_{\alpha}(a_{(1)}) \alpha(a_{(2)}) = \Cou(a) 1= \alpha(a_{(1)}) S_{\alpha}(a_{(2)}), \;\; a \in \Alg.\label{twistedantipoderelation}\end{equation}
\end{rlist}
\end{prop}
\begin{proof}
All the above follow easily from the definitions. For example to show (iii) we use Definition \ref{Def:admissible} (iii), the fact that $\alpha$ is a  bijection and the uniqueness of a counit on a Hopf algebra; to show (iv) we use again Definition \ref{Def:admissible} (iii),  the fact that $\alpha$ is a homomorphic bijection, property (ii) and the uniqueness of an antipode on a Hopf algebra. To obtain (vii) we apply the homomorphism $\alpha$ to the defining antipode relation and use the fact that as $\alpha$ `commutes' with the coproduct we have in Sweedler's notation  $\alpha(a_{(1)}) = \alpha(a)_{(1)}$ and  $\alpha(a_{(2)}) = \alpha(a)_{(2)}$.
\end{proof}

\subsection*{Cocycles and generating functionals}

In this subsection we discuss quickly the generating functionals and cocycles for representations of Hopf $^*$-algebras and basic relations between these notions.

\begin{deft}
Let $\Alg$ be a Hopf $^*$-algebra. A generating functional on $\Alg$ is a functional $L:\Alg \to \bc$ which is conditionally positive, hermitian and vanishes at $1$:
\[ L(a^*a) \geq 0, \;\; a \in K_1,\]
\[ L(b^*) = \overline{L(b)}, \;\; b \in \Alg,\]
\[ L(1)=0.\]
\end{deft}

The importance of generating functionals lies in the fact that they generate convolution semigroups of states (see \cite{Schurmann} and \cite{LS} for the topological, analytic context), and thus further classify quantum L\'evy processes up to stochastic equivalence (see \cite{Schurmann}).

Let $D$ be a pre-Hilbert space. By a representation of $\Alg$ on $D$ we will always mean a unital $^*$-homomorphism from $\Alg$ to $\mathcal{L^{\dagger}}(D)$, the $^*$-algebra of operators on $D$ which admit (pre-)adjoints defined on $D$ and leaving $D$ invariant. Note that if $\Alg$ is a CQG algebra each such representation is bounded operator valued, and hence extends automatically to a unital $^*$-homomorphism from $\Alg$ to $B(\Hil)$, where $\Hil$ is the Hilbert space completion of $D$.

\begin{deft}
Let $\Alg$ be a Hopf $^*$-algebra, and let $D$ be a pre-Hilbert space. A linear map $\eta:\Alg\to D$ is said to be a (nondegenerate) cocycle (for a representation $\pi$ of $\Alg$ on $D$) if it satisfies the equation:
\begin{equation} \label{cocyclerelation} \eta(ab)= \pi(a) \eta(b) + \eta(a) \Cou(b),\;\;\; a, b \in \Alg\end{equation}
and the image $\eta(\Alg)$ is dense in $D$.
\end{deft}

Usually non-degeneracy is not a part of the defining condition for cocycles, but for us it is convenient to include it, as we will work only with the cocycles which have this property (in any case a degenerate cocycle $\eta$ can be always viewed as a non-degenerate one by restricting the representation $\pi$ to the invariant subspace $\eta(\Alg)$). Note also that as for each cocycle $\eta(1)=0$, we have $\eta(\Alg)= \eta(K_1)$.

\begin{deft} \label{deft:yieldscoboundary}
Let $L:\Alg \to \bc$ be a generating functional and let $\eta:\Alg \to D$ be a cocycle. We say that $\eta$ yields the coboundary of $L$ if
\begin{equation} \label{formula:yieldscoboundary} L(ab) = \Cou(a) L(b) + L(a) \Cou (b) + \la \eta(a^*), \eta(b) \ra, \;\;\; a, b \in \Alg\end{equation}
(equivalently, $L(ab) = \la \eta(a^*), \eta(b) \ra$ for $a, b \in K_1$).
\end{deft}

We say that two cocycles $\eta_1:\Alg \to D_1$, $\eta_2:\Alg \to D_2$ are \emph{unitarily equivalent} if there exists a unitary operator $U:\Hil_1 \to \Hil_2$, where $\Hil_1, \Hil_2$ denote the respective Hilbert space completions, such that for all $a \in \Alg$ we have $\eta_2(a) = U \eta_1(a)$.
It is easy to see that if two cocycles yield the coboundary of the same generating functional then they are unitarily equivalent. Moreover if $\eta: \Alg \to D$ is a cocycle and $L:\Alg \to \bc$ is a functional satisfying \eqref{formula:yieldscoboundary} then $L$ is a generating functional if and only if it is hermitian. The following result forms a part of the Sch\"urmann Reconstruction Theorem and can be shown via a GNS-type construction.

\begin{prop}[\cite{Schurmann}] \label{funct->cocycle}
Let $L:\Alg \to \bc$ be a generating functional. Then there exists a cocycle $\eta:\Alg \to D$ which yields the coboundary of $L$.
 \end{prop}

Thus it is natural to ask the following question, which goes back to \cite{SchuProc}.

\begin{quest} \label{mainquest}
 Given a cocycle $\eta$ does it admit a generating functional $L$ for which it yields the coboundary?
 \end{quest}

In general it is known that the answer to this question is negative, even for cocycles with respect to trivial representations (i.e.\ for a multiple of the counit)  -- see Example 2.1 in \cite{Skeide}. The construction in \cite{Skeide} can be modified to yield counterexamples for a $CQG$-algebra. The answer to Question \ref{mainquest} is positive for all cocycles  on $\bc[\mathbb{F}_n]$, where $\mathbb{F}_n$ is the free group on $n$ generators, and on the so-called Brown-Glockner-von Waldenfels $^*$-bialgebra $\mathcal{U}\la n\ra$ (\cite{SchuProc}).

The terminology used above  (i.e.\ the cocycle, the coboundary) is related to the homological viewpoint on the relation \eqref{formula:yieldscoboundary}, which we discuss next.

\begin{lem} \label{Lemma-definitionphi}
If $\eta:\Alg \to D$ is a cocycle then there exists a unique linear map $\varphi:K_1\otimes_{\Alg} K_1\to \mathbb{C}$ such that
\begin{equation}\label{eq-def-phi}
\varphi(a\otimes b) = \langle\eta(a^*),\eta(b)\rangle, \;\;\; a, b \in K_1.
\end{equation}
\end{lem}
\begin{proof}
Denote the representation for which $\eta$ is a cocycle by $\pi$. It is clear by linearity that the map $\tilde{\varphi}:K_1\otimes K_1 \to\mathbb{C}$,
\begin{equation} \label{defvarphi}
\tilde{\varphi}(a\otimes b) = \langle \eta(a^*), \eta(b)\rangle, \;\; a, b \in K_1
\end{equation}
is well-defined. To get $\varphi$, we show that for $a,c\in K_1$, $b\in \Alg$, we have
\begin{eqnarray*}
\tilde{\varphi}(ab\otimes c) &=& \langle \eta(b^*a^*), \eta(c)\rangle = \langle\pi(b)^* \eta(a^*), \eta(c)\rangle \\
&=& \langle \eta(a^*),\pi(b) \eta(c)\rangle = \langle \eta(a^*), \eta(bc)\rangle \\
&=& \tilde{\varphi}(a\otimes bc).
\end{eqnarray*}
\end{proof}

We can naturally view $\bc$ as an $\Alg$-bimodule with both actions given by the counit. Then the coboundary of a functional $\omega: \Alg \to \bc$ is the map $\partial \omega: \Alg \ot \Alg \to \bc$ defined as $\partial \omega(a \ot b) = - \omega(ab) + \Cou(a) \omega(b) + \omega(a) \Cou(b)$, $a, b \in \Alg$.  Consider  the exact sequence
\[
0 \to H_2(\Alg,\mathbb{C})\to K_1\otimes_{\Alg} K_1 \to K_1 \to H_1(\Alg,\mathbb{C})\to 0
\]
stated for group algebras in \cite[Lemma 5.6]{netzer+thom13}, but holding more generally for unital $*$-algebras with a character, as shown in \cite{UweAndreas} ($K_1\otimes_{\Alg} K_1$ denotes the tensor product of $\Alg$-modules over $\Alg$).
Earlier remarks and simple observations imply that a functional $L:\Alg\to \mathbb{C}$ for which $\eta$ yields the coboundary
exists if and only if $\varphi$  vanishes on the kernel of the multiplication map from $K_1\otimes_{\Alg} K_1\to K_1$. The exactness of the displayed sequence means that one can interpret this in terms of $H_2(\Alg,\mathbb{C})$.
Furthermore $L$ is determined by $\varphi$ up to a linear functional on $H_1(\Alg,\mathbb{C})$.

Finally note that given a cocycle $\eta:\Alg \to D$ the map $\tilde{\varphi}:K_1\otimes K_1 \to \mathbb{C}$ defined by \eqref{defvarphi} is a $2$-cocycle, since for any $a, b, c \in K_1$
\begin{eqnarray*}
\partial\tilde{\varphi}(a\otimes b\otimes c) &=&
\varepsilon(a) \langle\eta(b^*),\eta(c)\rangle - \big\langle\eta\big((ab)^*\big),\eta(c)\big\rangle + \langle\eta(a^*),\eta(bc)\rangle - \langle\eta(a^*),\eta(b)\rangle\varepsilon(c) \\
&& = - \langle \pi(b)^*\eta(a^*) \eta(c)\rangle +  \langle \eta(a^*), \pi(b)\eta(c)\rangle= 0.
\end{eqnarray*}

\section{Main result} \label{MainresultSection}

As mentioned in the introduction, in the article \cite{Kyed} Kyed, following the unpublished notes of Vergnioux,  showed that Question \ref{mainquest} has a positive answer if the cocycle $\eta$ is \emph{real}, i.e.\ when it satisfies a certain symmetry relation with respect to the antipode (see Definition 4.1 in \cite{Kyed}, but note a difference resulting from another convention for scalar products). In this main section we extend this result, considering an $\alpha$-twisted reality condition (see Theorem \ref{mainthm}). Throughout the section \textbf{we fix a Hopf $^*$-algebra $\Alg$ and an admissible bijection $\alpha:\Alg \to \Alg$ and write $\gamma:=(\id + \alpha):\Alg \to \Alg$}.

The first lemma generalizes Lemma 4.8 in \cite{Kyed} (due to Vergnioux).

\begin{lem} \label{cocyclealpha}
Let $\eta: \Alg \to D$ be a cocycle for a representation $\pi$. Then the following equalities hold for all $a \in \Alg$:
\begin{rlist}
\item $\eta(S_\alpha(a)) = - \pi(S_{\alpha}(a_{(1)})) \eta(\alpha(a_{(2)})),$
\item $ \eta(\alpha(a)) = - \pi(\alpha(a_{(1)})) \eta(S_\alpha(a_{(2)})),$
\item $\eta(S_\alpha(a)^*) = - \pi(S_{\alpha}(a_{(2)}))^* \eta(\alpha(a_{(1)})^*),$
\item $ \eta(\alpha(a)^*) = - \pi(\alpha(a_{(2)}))^* \eta(S_\alpha(a_{(1)})^*),$
\end{rlist}
\end{lem}
\begin{proof}
It suffices to apply $\eta$ (or $\eta \circ ^*$) to the twisted antipode relations \eqref{twistedantipoderelation} and use the fact that $\eta(1)=0$ together with the counit relation.
\end{proof}

\begin{deft}
We say that a generating functional $L:\Alg \to \bc$ is $S_{\alpha}$-invariant if $L\circ S_{\alpha} = L$.
\end{deft}

The next lemma establishes the  algebraic property of the cocycle associated to an $S_{\alpha}$-invariant generating functional (for $\alpha=\id$ it is the reality condition of Kyed, mentioned above).

\begin{lem} \label{Lemmacocyclereal}
Let $L:\Alg \to \bc$ be an $S_{\alpha}$-invariant generating functional and assume that $\eta: \Alg \to D$ is a cocycle which yields the coboundary of $L$. Then
for any $a, b \in \Alg$
\[ \langle \eta(a), \eta(b) \ra = \la \eta (S_{\alpha} (b)^*), \eta(S_{\alpha}(a^*)) \ra. \]
\end{lem}
\begin{proof}
Choose $a, b \in \Alg$ and compute
\begin{align*} \ol{\Cou(a)} L(b) + \ol{L(a)} \Cou(b) &+ \la \eta(a), \eta(b) \ra = L(a^*b) = L(S_{\alpha}(a^*b)) = L(S_{\alpha}(b) S_{\alpha}(a^*))
\\&= \Cou(S_{\alpha}(b)) L(S_{\alpha}(a^*)) + L(S_{\alpha}(b)) \Cou(S_{\alpha}(a^*)) + \la \eta (S_{\alpha} (b)^*), \eta(S_{\alpha}(a^*)) \ra
\\&= \Cou(b) \ol{L(a)} + L(b) \ol{\Cou(a)} + \la \eta (S_{\alpha} (b)^*), \eta(S_{\alpha}(a^*)) \ra.
\end{align*}
\end{proof}

This motivates the following definition.

\begin{deft}  \label{defrealcocycle}
We say that a cocycle  $\eta: \Alg \to D$ is $\alpha$-real if
\begin{equation} \label{defrealcocycleformula} \langle \eta(a), \eta(b) \ra = \la \eta (S_{\alpha} (b)^*), \eta(S_{\alpha}(a^*)) \ra, \;\;\;a, b \in \Alg. \end{equation}
\end{deft}

In particular when $\alpha=\id$ we recover the notion of a real cocycle introduced by Vergnioux and Kyed.

\begin{lem} \label{Lemmacocyclerealformula}
Let $L:\Alg \to \bc$ be an $S_{\alpha}$-invariant generating functional and assume that $\eta: \Alg \to D$ is a cocycle which yields the coboundary of $L$. Then
for any $a \in \Alg$
\begin{equation} \label{definingformula} L(\gamma(a))= - \la \eta(S_{\alpha}(a_{(1)})^*), \eta(\alpha(a_{(2)})) \ra = - \la \eta(\alpha(a_{(1)})^*), \eta(S_{\alpha}(a_{(2)})) \ra.\end{equation}
\end{lem}

\begin{proof}
Let $a\in \Alg$ and apply $L$ to the twisted antipode relation \eqref{twistedantipoderelation}. This yields
\[L(  S_{\alpha}(a_{(1)}) \alpha(a_{(2)}) ) =  0 = L(\alpha(a_{(1)}) S_{\alpha}(a_{(2)}))\]
and further via \eqref{formula:yieldscoboundary}
\[0 = \Cou(S_{\alpha}(a_{(1)}))L(\alpha(a_{(2)}) ) +  L(S_{\alpha}(a_{(1)}))\Cou(\alpha(a_{(2)})) + \la \eta(S_{\alpha}(a_{(1)})^*), \eta(\alpha(a_{(2)})) \ra, \]
\[0 = \Cou(\alpha(a_{(1)})) L(S_{\alpha}(a_{(2)})) + L(\alpha(a_{(1)})) \Cou( S_{\alpha}(a_{(2)})) + \la \eta(\alpha(a_{(1)})^*), \eta(S_{\alpha}(a_{(2)})) \ra.\]
Using the counit relation and the fact that both $L$ and $\Cou$ are invariant under $S_\alpha$ yields
\[0 = L(\alpha(a)) + L(a) + \la \eta(S_{\alpha}(a_{(1)})^*), \eta(\alpha(a_{(2)})) \ra,\]
\[ 0 = L(a) +  L(\alpha(a)) +\la \eta(\alpha(a_{(1)})^*), \eta(S_{\alpha}(a_{(2)})) \ra,\]
which ends the proof.
\end{proof}

In fact the equality of the two formulas appearing above is a general fact.

\begin{lem}\label{twoformulaone}
Let $\eta: \Alg \to D$ be a cocycle. Then for any $a \in \Alg$
 \[ \la \eta(S_{\alpha}(a_{(1)})^*), \eta(\alpha(a_{(2)})) \ra = \la \eta(\alpha(a_{(1)})^*), \eta(S_{\alpha}(a_{(2)})) \ra.\]
\end{lem}
\begin{proof}
Say that $\eta$ is a cocycle for a representation $\pi$.
It suffices to pick $a \in \Alg$ and use consecutively relations (ii) and (iii) of Lemma \ref{cocyclealpha}:
\begin{align*} \la \eta(S_{\alpha}(a_{(1)})^*), & \eta(\alpha(a_{(2)})) \ra  = - \la \eta(S_{\alpha}(a_{(1)})^*), \pi(\alpha(a_{(2)})) \eta(S_\alpha(a_{(3)})) \ra
\\ &= - \la \pi(\alpha(a_{(2)}))^* \eta(S_{\alpha}(a_{(1)})^*),  \eta(S_\alpha(a_{(3)})) \ra=
\la \eta(\alpha(a_{(1)})^*), \eta(S_{\alpha}(a_{(2)})) \ra.
 \end{align*}
\end{proof}

\begin{tw}\label{cocycle->funct}
Let $\Alg$ be a Hopf $^*$-algebra, $\alpha:\Alg \to \Alg$ be an admissible bijection and let $\eta$ be an $\alpha$-real cocycle on $\Alg$. Then the formula \eqref{definingformula} defines an $S_{\alpha}$-invariant generating functional $L$ such that $\eta$ yields the coboundary of $L$.
\end{tw}
\begin{proof}
Note first that both expressions on the right hand side of \eqref{definingformula} coincide due to Lemma \ref{twoformulaone}.

We divide the proof into several parts. First we show that $L$ is hermitian. To that end fix $a$ in $\Alg$ and put $b= \gamma^{-1}(a)$. Then $a^* = \gamma(b)^* = b^* + \alpha(b)^* = \gamma(c)$, where $c = \alpha(b)^* = \alpha^{-1}(b^*)$. This means in particular that $\alpha(c_{(1)})^*= b_{(1)}$ and $c_{(2)}^* = \alpha(b_{(2)})$. Thus
\begin{align*} L(a^*) &= L(\gamma(c)) = - \la \eta(\alpha(c_{(1)})^*), \eta(S_{\alpha}(c_{(2)})) \ra = - \la \eta(b_{(1)}), \eta(S_{\alpha}(c_{(2)})) \ra
\\&=  - \la \eta(S_{\alpha} (S_{\alpha} (b_{(1)})^*)^*), \eta(S_{\alpha}(c_{(2)}))  \ra =
 - \la \eta(c_{(2)}^*) , \eta(S_{\alpha} (b_{(1)})^*) \ra
 \\& = - \la \eta(\alpha(b_{(2)})) , \eta(S_{\alpha} (b_{(1)})^*) \ra = \overline{ - \la \eta(S_{\alpha} (b_{(1)})^*) , \eta(\alpha(b_{(2)}))  \ra}= \overline{L(a)},
 \end{align*}
where we first used the second equality in \eqref{definingformula}, then Proposition \ref{propadmissbij} (v), then the fact that $\eta$ is $\alpha$-real and finally the  first equality in  \eqref{definingformula}.

In the second step we show that $\eta$ yields the coboundary of $L$. Consider first an auxiliary functional $L'=L\circ \gamma$. Let $a, b \in \Alg$. Then applying the second equality in \eqref{definingformula} and the fact that $\alpha$ is a homomorphism we obtain
\[ L'(ab) = - \langle\eta\big(\alpha(b_{(1)})^*\alpha(a_{(1)})^*\big),\eta\big(S_\alpha(b_{(2)})S_\alpha(a_{(2)}\big)\rangle,
\]
which, when we declare that $\eta$ is a cocycle for a representation $\pi$, and  use the  cocycle property \eqref{cocyclerelation} leads to
\begin{align*}
L'(ab)=&-  \left\langle\eta\big(\alpha(a_{(1)})^*\big), \pi\big(\alpha(b_{(1)})S_\alpha(b_{(2)})\big)\eta\big(S_\alpha(a_{(2)})\big)\right\rangle
\\& -
\left\langle\eta\big(\alpha(a_{(1)})^*\big), \pi\big(\alpha(b_{(1)})\big)\eta\big(S_\alpha(b_{(2)})\big)\right\rangle\Cou(a_{(2)})
\\& - \Cou(a_{(1)}) \left\langle\eta\big(\alpha(b_{(1)})^*\big), \pi\big(S_\alpha(b_{(2)})\big)\eta\big(S_\alpha(a_{(2)})\big)\right\rangle
\\& - \Cou(a_{(1)}) \Cou(a_{(2)})\left\langle\eta\big(\alpha(b_{(1)})^*\big),\eta\big(S_{\alpha}(b_{(2)})\big)\right\rangle.
\end{align*}
Further via the twisted antipode relation \eqref{twistedantipoderelation}, Lemma \ref{cocyclealpha} (ii) and (iii) and  the counit relation used in the first equality we obtain
\begin{align*}
L'(ab)=& - \Cou(b)\left\langle\eta\big(\alpha(a_{(1)})^*\big),\eta\big(S_{\alpha}(a_{(2)})\big)\right\rangle +
\left\langle\eta\big(\alpha(a)^*\big),\eta\big(\alpha(b)\big)\right\rangle +  \left\langle\eta\big(S_{\alpha}(b)^*\big),\eta\big(S_{\alpha}(a)\big)\right\rangle
\\&  -
 \Cou(a)\left\langle\eta\big(\alpha(b_{(1)})^*\big),\eta\big(S_{\alpha}(b_{(2)})\big)\right\rangle = L'(a)\Cou(b)  + \left\langle\eta\big(\alpha(a)^*\big),\eta\big(\alpha(b)\big)\right\rangle
 \\&+ \left\langle\eta\big(S_{\alpha}(b)^*\big),\eta\big(S_{\alpha}(a)\big)\right\rangle  + \Cou(a)L'(b)
\end{align*}
which in the end via the $\alpha$-reality of $\eta$ means that
\[L'(ab)= \Cou(a)L'(b) + L'(a)\Cou(b) + \left\langle\eta\big(\alpha(a)^*\big),\eta\big(\alpha(b)\big)\right\rangle + \left\langle\eta(a^*),\eta(b)\right\rangle.
\]
Therefore, using the language of Section \ref{NotPrel},  we obtain
\begin{align*}
\partial L'(a\otimes b) &= \Cou(a)L'(b) - L'(ab) + L'(a)\Cou(b) = - \langle\eta(a^*),\eta(b)\rangle -  \left\langle\eta\big(\alpha(a)^*\big),\eta\big(\alpha(b)\big)\right\rangle
\\
&= - \big(\tilde{\varphi}\circ({\rm id}\otimes{ \id}+\alpha\otimes\alpha)\big)(a\otimes b),
\end{align*}
for all $a,b\in \Alg$, i.e.\
\[
\partial L' = - \tilde{\varphi}\circ({\rm id}\otimes{\rm id}+\alpha\otimes\alpha).
\]

Denote by $d:\Alg\otimes \Alg\to \Alg$ the (linear) boundary operator,
\[
d(a\otimes b) = \varepsilon(a)b -ab+a\varepsilon(b), \;\;\; a, b \in \Alg.
\]
We then have for any linear functional  $\omega:\Alg\to\mathbb{C}$ the equality $\partial \omega = \omega\circ d$.
Since $\alpha$ is an $\Cou$-preserving homomorphism, we have also $
\alpha\circ d = d\circ(\alpha\otimes\alpha)$.
Thus
\begin{align*}
(\partial L)\circ({\rm id}\otimes {\rm id}+\alpha\otimes\alpha) =& L\circ d \circ({\rm id}\otimes {\rm id}+\alpha\otimes\alpha)
\\=& L\circ ({\rm id}+\alpha)\circ d  =\partial \big(L\circ({\rm id}+\alpha)\big) =
 -\tilde{\varphi}\circ({\rm id}\otimes {\rm id}+\alpha\otimes\alpha).
\end{align*}
Using finally the last defining property of the admissible bijection $\alpha$ we obtain the desired equality
\[
\partial L= -\tilde{\varphi},
\]
equivalent to the fact that $\eta$ yields the coboundary of $L$. As explained after Definition \ref{deft:yieldscoboundary} together with the fact that $L$ is hermitian it implies that $L$ is a generating functional.

It remains to show that $L$ is $S_{\alpha}$ invariant. Once again fix $a$ in $\Alg$ and put $b= \gamma^{-1}(a)$. Then $S_{\alpha}(a) = \gamma (S_{\alpha}(b))$ as $\alpha$ commutes with $S_{\alpha}$ and
\begin{align*} L(S_\alpha(a)) &= - \la \eta(S_{\alpha}(S_\alpha (b_{(2)}))^*), \eta(\alpha(S_\alpha(b_{(1)}))) \ra =
 - \la \eta(S_{\alpha}(S_\alpha (b_{(2)}))^*), \eta(S_\alpha(\alpha(b_{(1)}))) \ra =
 \\&= - \la \eta(\alpha(b_{(1)})^*),  \eta(S_\alpha (b_{(2)})) \ra = L(a),
\end{align*}
where we first used the first  equality in \eqref{definingformula} and Proposition  \ref{propadmissbij} (vi), then the fact that $\alpha$ commutes with $S_\alpha$, then $\alpha$-reality of $\eta$ and finally the  second equality in  \eqref{definingformula}. Note that this last result, i.e.\ the $S_{\alpha}$-invariance for the functional $L$ given by \eqref{definingformula} in the case of $\alpha=\id$ is Proposition 7.22 of \cite{DawFimSkaWhi}.

\end{proof}

The following theorem explains the title of the paper.

\begin{tw}\label{mainthm}
Let $\Alg$ be a Hopf $^*$-algebra and let $\alpha:\Alg \to \Alg$ be an admissible bijection. There exists a one-to-one correspondence, up to a unitary equivalence on the cocycle part, between $\alpha$-real cocycles $\eta$ on $\Alg$ and $S_{\alpha}$-invariant generating functionals $L$ on $\Alg$, so that given an $S_{\alpha}$-invariant  $L$ the corresponding $\alpha$-real cocycle $\eta$ yields its coboundary.
\end{tw}
\begin{proof}
This is a straightforward consequence of Theorem \ref{cocycle->funct}, Proposition \ref{funct->cocycle} and comments before it and Lemma \ref{Lemmacocyclereal}.
\end{proof}

We list here an immediate corollary, modelled on Corollary 4.9 in \cite{FabioUweAnna}. It shows that from our point of view among the admissible bijections given by the scaling automorphism group of a compact quantum group $\tau_{\frac{i}{2}}$, the one leading to the symmetry under the unitary antipode, plays a distinguished role.

\begin{cor}
Let $\QG$ be a compact quantum group and let $t\in \br \setminus \{\frac{1}{2}\}$. Suppose that a cocycle $\eta:\PolQG \to D$ is $\tau_{it}$-real. Then it is $\tau_{is}$-real for any $s \in \br$.
\end{cor}
\begin{proof}
Assume that the cocycle $\eta$ is as above. Then its associated (via Theorem \ref{cocycle->funct}) generating functional $L$ satisfies the formula
\[L \circ S \circ \tau_{it} = L.\]
Applying this formula twice yields $L= L \circ \tau_{2it - i}$. Then the same argument as in Corollary 4.9 in \cite{FabioUweAnna} (see also the proofs of Lemma 2.9 (2) in \cite{Reiji} and Proposition 3.5 in \cite{FST}) imply that in fact $L= L \circ S \circ \tau_z$ for any $z \in \bc$. Fixing $s \in \br$ and applying Theorem \ref{mainthm} for $\eta$, $L$ and $\alpha= \tau_{is}$ ends the proof (recall again that our convention for the scaling group is that of \cite{DawFimSkaWhi}, not of \cite{wor2}).
\end{proof}

In particular the above corollary reveals that if a cocycle is $\tau_{it}$-real for $t\neq \frac{1}{2}$, then it is real in the sense of Vergnioux and Kyed. Example 11.5 of \cite{FabioUweAnna} shows that the assumption $t \neq \frac{1}{2}$   is indeed necessary.

\section{Applications}
In this section we discuss the applications of the main results of Section \ref{MainresultSection}. Both concern quantum groups: the first is related to decomposing quantum L\'evy processes on a compact quantum group, whereas the second treats a generalization of a theorem from \cite{DawFimSkaWhi} characterising the Haagerup property for a discrete quantum group via the existence of a proper \emph{real} cocycle on its dual.

Throughout this Section \textbf{we fix a compact quantum group $\QG$ and an admissible bijection $\alpha:\PolQG \to \PolQG$}. As explained earlier, for $\Alg=\PolQG$ we can view all representations of $\Alg$ as bounded operator valued, so that we will also speak of Hilbert (rather than pre-Hilbert) space valued cocycles. Given $\QG$ we will write $K_2:=\tu{Lin} \{ab: a, b \in K_1\}$.

\subsection{Extracting maximal Gaussian parts of quantum L\'evy processes}

We need to introduce a few more notations and concepts. A \emph{quantum L\'evy process} on $\QG$ is a counterpart of a classical concept of a L\'evy process on a compact group; we refer for specific definitions to \cite{Schurmann} and here only stress the fact that there is a natural one-to-one correspondence between (stochastic equivalence classes of) quantum L\'evy processes on $\QG$  and generating functionals on $\Pol(\QG)$. A generating functional $L$ (and, by extension, its associated quantum L\'evy process) is called \emph{Gaussian} if $L(a^*a) = 0$ for all $a \in K_2$. This is equivalent to the fact that the cocycle $\eta$ associated to such $L$ via Proposition \ref{funct->cocycle} is a cocycle with respect to the trivial representation of $\PolQG$:
\begin{equation} \label{Gaussiancocycle} \eta(ab) = \Cou(a) \eta(b) + \eta(a) \Cou(b), \;\;\;a,b \in \PolQG. \end{equation}
In fact any map $\eta:\Pol(\QG) \to \Hil$, where $\Hil$ is a Hilbert space, satisfying the equation \eqref{Gaussiancocycle} is called a \emph{Gaussian cocycle}.

Consider now any cocycle $\eta:\Pol(\QG) \to \Hil$ with an associated representation $\pi: \PolQG \to B(\Hil)$. Define $\Ril= \ol{\eta(K_2)}$ and $\Gil=\bigcap_{a\in \PolQG}\, \Ker(\pi(a) - \Cou(a) I_{\Hil})$. Then Chapter 5 of \cite{Schurmann} (see also \cite{SchuProc}) shows that the following facts hold: both subspaces $\Ril$ and $\Gil$ are left invariant by operators in $\pi(\PolQG)$, $\Hil = \Ril \oplus \Gil$ (in the Hilbert space sense -- so $\Ril \perp \Gil$), if we denote by $P_R$ and $P_G$ the respective orthogonal projections then
$\eta_R:= P_R \circ \eta$ is a cocycle on $\PolQG$ and $\eta_G=P_G \circ \eta$ is a Gaussian cocycle on $\PolQG$ (both with respect to the respective restrictions of $\pi$). We say that $\eta_G$ is a \emph{maximal Gaussian part} of $\eta$. If $\Gil=\{0\}$ we say that $\eta$ is \emph{purely non-Gaussian}. In particular it is easy to see that $\eta_R$ is purely non-Gaussian.

\begin{deft}
 We say that a quantum L\'evy process   admits a decomposition into a maximal Gaussian part and a purely non-Gaussian part if its generating functional $L:\PolQG\to \bc$ decomposes as $L=L_G + L_R$, where $L_G, L_R:\PolQG\to \bc$ are generating functionals such that the associated (via Proposition \ref{funct->cocycle}) cocycles are respectively Gaussian and purely non-Gaussian.
\end{deft}

For motivations behind this terminology, relations to the classical L\'evy-Khintchine decomposition and the (non)-uniqueness of a possible decomposition described above we refer to \cite{Schurmann}. Here we only remark that if the answer to Question \ref{mainquest} were always positive (which we know is not true) then the decomposition defined above would always exist. The decomposition exists for all quantum L\'evy processes on classical groups and for  $SU_q(2)$ (\cite{SchSk}) or, more generally $SU_q(N)$ for arbitrary $N \in \bn$ (\cite{FKLS}). For the non-existence examples we refer to the forthcoming work \cite{UweAndreas}; here we will prove, using the results of Section 2, that the decomposition exists in presence of $\alpha$-symmetry.

We begin with two lemmas, in which we use several times the notion of an adjoint for a densely defined \emph{conjugate} linear Hilbert space operator.

\begin{lem}
Let $\eta:\PolQG\to \Hil$ be an $\alpha$-real cocycle and let $D=\eta(\PolQG)$. Then the formulas
\[
T(\eta(a)) = \eta\big(S_{\alpha}(a)^*\big), \;\;\; T'(\eta(a)) = \eta\big(S_{\alpha}(a^*)\big), \;\;\; a \in \PolQG,
\]
define conjugate linear involutive maps $T,T':\eta(A)\to \Hil$. The operators $T$ and $T'$ are mutually adjoint on $D$ and therefore closable as densely defined operators on $\Hil$. Furthermore each of the maps $T, T'$ leaves $\eta(K_2)$ invariant and the corresponding restrictions $T_R:=T|_{\eta(K_2)}$, $T'_R = T'|_{\eta(K_2)}$ are densely defined conjugate linear closable operators on the Hilbert space $\Ril:=\ol{\eta(K_2)}$.
\end{lem}
\begin{proof}
To verify that $T$ is well defined we need to check that $\eta(a)=0$ implies $T(\eta(a))=\eta\big(S_{\alpha}(a)^*\big)=0$ for any $a \in \PolQG$. It is in fact sufficient to do this for $a\in K_1$ (as $S_{\alpha}$ is unital and $\Cou$-preserving and $\eta(1)=0$). Let $L$ be the unique $S_{\alpha}$-invariant generating functional associated to $\eta$ via Theorem \ref{cocycle->funct}. Then the condition $\eta(a)=0$ implies that  $L(ba) = \la \eta(b^*), \eta(a) \ra =  0$ for all
 $b\in K_1$.  Since $L$ is $S_{\alpha}$-invariant,
\[0 = \overline{L\circ S_\alpha(ba)} = L\big( S_\alpha(b)^*S_\alpha(a)^*\big) = \la \eta(S_\alpha(b)), \eta(S_\alpha(a)^*) \ra
\]
for all $b\in K_1$, which by non-degeneracy of $\eta$ and the fact that $S_{\alpha}$ is bijective is equivalent to $\eta\big(S_\alpha(a)^*\big)=0$. The proof for $T'$ is identical. The fact that both $T$ and $T'$ are involutive follows from Proposition \ref{propadmissbij} (v).

Further since $\eta:A\to H$ is $\alpha$-real
\[
\langle T(\eta(a)),T'(\eta(b))\rangle = \big\langle\eta\big(S_\alpha(a)^*\big),\eta\big(S_\alpha(b^*)\big)\big\rangle
= \langle \eta(b),\eta(a)\rangle
\]
for $a,b\in \PolQG$. This means that $T'=T^*|_D$, $T=(T')^*|_D$ and implies closability of both operators.


Finally since  $S_{\alpha}$ as an $\Cou$-preserving anti-homomorphism  leaves $K_2$ invariant, $T$ and $T'$ map $\eta(K_2)$ to itself and the last statement follows easily.
\end{proof}

In view of the above we will write in what follows $T^{\dagger}$ instead of $T'$.

\begin{lem}\label{lem-T-P}
Let $\eta:\PolQG\to \Hil$ be an $\alpha$-real cocycle. Then each of the cocycles
 $\eta_R = P_R\circ \eta$  and $ \eta_G = P_G\circ \eta$ is also $\alpha$-real.
\end{lem}

\begin{proof}
We use the same notation as in the last lemma. Begin by denoting the closure of $T$ by $\ol{T}$, so that we have $\ol{T}=(T^{\dagger})^*$, and similarly let $\ol{T_R}$ denote the closure of $T_R$, so that $\ol{T_R}=(T^{\dagger}_R)^*$.  Note that we have of course $\ol{T_R}^* = T_R^*$ and as $(T_R^{\dagger})^2= \id_{\eta(K_2)}$, we also have $(T_R^*)^2= \id_{\tu{Dom} T_R^*}$.

Take $\zeta\in D:=\eta(\PolQG)$ and $\xi \in \eta(K_2)$ and compute
\[ \la P_R \zeta, T_R^{\dagger} \xi \ra = \la \zeta,  T_R^{\dagger} \xi \ra = \la \zeta, T^{\dagger} \xi \ra = \la  \xi, T\zeta \ra = \la  P_R \xi, T \zeta \ra =
\la  \xi, P_R T \zeta \ra,\]
which proves  $P_R (D) \subset \tu{Dom} (\ol{T_R})$, $\ol{T_R}P_R|_D= P_R T$, with the latter equality understood as one for linear operators acting between $D$ and $\Ril$. Similarly we show $P_R (D) \subset \tu{Dom} (T_R^*)$ and $T_R^*P_R|_D = P_R T^{\dagger}$. We will need another of these relations, namely $T_R^* (P_R(D))\subset \tu{Dom} (T_R^*)$ and $(T_R^*)^2 P_R|_D = P_R|_D$. To establish it we pick again $\zeta \in D$ and $\xi \in \tu{Dom} (T_R)=\eta(K_2)$ and write
\[ \la  T_R^* P_R \zeta, T_R \xi \ra = \la  P_R T^{\dagger} \zeta,  T_R \xi \ra = \la   T^{\dagger} \zeta,  P_R T_R \xi \ra = \la   T^{\dagger} \zeta,   T \xi \ra =
\la  T^2 \xi, \zeta \ra = \la \xi, \zeta \ra, \]
 which shows the desired equalities.

We are now ready for the algebraic computation. For all $a, b \in \PolQG$
\begin{align*}
\big\langle\eta_R\big(S_\alpha(a)^*\big),&\eta_R\big(S_\alpha(b^*)\big)\big\rangle = \big\langle P_R\eta\big(S_\alpha(a)^*\big),P_R\eta\big(S_\alpha(b^*)\big)\big\rangle = \langle P_RT\eta(a),P_RT^{\dagger}\eta(b)\rangle
\\&= \langle \ol{T_R}P_R(\eta(a)),T_R^*P_R(\eta(b))\rangle
= \langle (T_R^*)^2 P_R(\eta(b),) P_R\eta(a)\rangle = \langle \eta_R(b),\eta_R(a)\rangle,
\end{align*}
which proves that $\eta_R$ is $\alpha$-real. The proof for the Gaussian part $\eta_G$ follows now from the (orthogonal) decomposition $\eta= \eta_R + \eta_G$.
\end{proof}

We now formulate the main theorem of this subsection.

\begin{tw}
Let $\QG$ be a compact quantum group, let $\alpha:\Pol(\QG) \to \Pol(\QG)$ be an admissible bijection. Then every quantum L\'evy process on $\QG$ whose generating functional is   $S_{\alpha}$-invariant allows the decomposition into a maximal Gaussian part and a purely non-Gaussian part.
\end{tw}
\begin{proof}
Let $L:\PolQG \to \bc$ be an $S_{\alpha}$-invariant  generating functional with an $\alpha$-real cocycle $\eta$ associated to it by Proposition \ref{funct->cocycle}.  Lemma \ref{lem-T-P} and Theorem \ref{cocycle->funct} imply that there exist generating functionals $L_G,L_R:\PolQG\to\mathbb{C}$ for which $\eta_G$ and $\eta_R$ yield the respective coboundaries. Since the ranges of $\eta_G$ and $\eta_R$ are mutually orthogonal, we have
\begin{align*}
L(\gamma(a)) &= -\left\langle\eta\left(S_{\alpha}(a_{(1)})^*\right),\eta(\alpha(a_{(2)}))\right\rangle
\\&=  -\left\langle\eta_G\left(S_{\alpha}(a_{(1)})^*\right),\eta_G(\alpha(a_{(2)}))\right\rangle - \left\langle\eta_R\left(S_{\alpha}(a_{(1)})^*\right),\eta_R(\alpha(a_{(2)}))\right\rangle
 \\
&= L_G(\gamma(a)) + L_R(\gamma(a))
\end{align*}
for all $a\in \PolQG$.

\end{proof}

\subsection{Haagerup property for quantum groups via arbitrary proper cocycles}

The concept of the \emph{Haagerup property} for locally compact quantum groups was developed in \cite{DawFimSkaWhi} (we refer to that paper for motivations and history behind this notion). It was shown there  that a discrete quantum group  $\hat{\QG}$ has Haagerup property if its dual compact quantum group admits a \emph{proper real} cocycle. We will show below that  the same fact is true if we delete the adjective `real'. We first introduce some definitions.

  Let $\QGam$ be a discrete quantum group with the compact quantum group dual $\QG$. Recall the notations related to irreducible representations of $\QG$ introduced in Section 1. Each functional $\omega: \PolQG \to \bc$ can be identified with a family of matrices
$(\omega^{\beta})_{\beta \in \Irr_\QG}$, with $\omega^\beta \in M_{n_{\beta}}$ defined as $\omega^\beta = (\id_{M_{n_{\beta}}} \ot \omega)(U^{\beta})$; all of these matrices are self-adjoint if and only if $\omega$ is $S$-invariant. Similarly each cocycle $\eta: \PolQG \to \Hil$ can be viewed as a family of Hilbert space valued matrices
$(\eta^{\beta})_{\beta \in \Irr_\QG}$. Note that we can view each $\eta^{\beta}$ as an operator in $B(\bc^{n_{\beta}};\bc^{n_{\beta}} \ot \Hil)$. In particular we can consider matrices $(\eta^{\beta})^* \eta^{\beta} \in B(\bc^{n_{\beta}};\bc^{n_{\beta}})\approx M_{n_{\beta}}$. The following two definitions were introduced in \cite{DawFimSkaWhi} (Definition 7.16 and 7.20).

\begin{deft}
A cocycle $\eta: \PolQG \to \Hil$ is called proper if for any $M >0$ there exists a finite set $F \subset \Irr_\QG$ such that for any $\beta \in \Irr_{\QG} \setminus F$ we have $(\eta^{\beta})^* \eta^{\beta} \geq M I_{\bc^{n_{\beta}}}$. Further an $S$-invariant generating functional $L: \PolQG \to \Hil$ is called proper if
for any $M >0$ there exists a finite set $F \subset \Irr_\QG$ such that for any $\beta \in \Irr_{\QG} \setminus F$ we have $L^{\beta} \leq - M I_{\bc^{n_{\beta}}}$.
\end{deft}

A warning is in place: an opposite sign in the inequality for $L$ in comparison to the one appearing in Definition 7.16 of \cite{DawFimSkaWhi} is due to the fact that in that paper a different convention, motivated by  classical geometric group theory, was used for generating functionals -- there the authors worked with counterparts of conditionally \emph{negative} definite functions, whereas here we use the quantum stochastics community convention and treat conditionally \emph{positive} definite objects.

The following  result arises as a combination of Theorem \ref{cocycle->funct} and methods introduced in \cite{DawFimSkaWhi} and \cite{FabioUweAnna}.

\begin{tw}
A discrete quantum group  $\QGam$ has the Haagerup property if and only if $\Pol(\QG)$ (where $\QG$ is the compact quantum group dual of $\QGam$) admits a proper cocycle.
\end{tw}

\begin{proof}
Theorem 7.23 of \cite{DawFimSkaWhi} says that $\QGam$ has the Haagerup property if and only if $\PolQG$ admits a proper real cocycle. Assume then that $\eta:\PolQG \to \Hil$ is a proper cocycle (not necessarily real). In the first step we use the procedure described in Theorem 5.4 of \cite{FabioUweAnna} to symmetrize $\eta$. To that end we consider the opposite representation of $\PolQG$ on the conjugate Hilbert space $\bar{\Hil}$ and the cocycle $\ol{\eta}:\PolQG \to \bar{\Hil}$ given by the formula
$\ol{\eta} (a) = \iota \circ \eta(R(a^*))$ ($a \in \PolQG$) where $\iota: \Hil \to \ol{\Hil}$ is the canonical isomorphism. Recall that the opposite representation is defined by the formula $\pi^{\tu{op}}(a)(\iota(v))= \iota ((\pi\ot R)(a^*)v)$ for $a \in \PolQG$, $v \in \Hil$.

The proof of Theorem 5.4 in \cite{FabioUweAnna} shows that the cocycle $(\eta + \bar{\eta}): \PolQG \to \Hil \oplus \bar{\Hil}$ is \emph{KMS real} (i.e.\ $\tau_{\frac{i}{2}}$-real in the terminology used in our article). It is easy to see that $\eta + \bar{\eta}$ is a proper cocycle, as by the orthogonality of the Hilbert space decomposition for each $\beta \in \Irr_{\QG}$ we have \[((\eta+\bar{\eta})^{\beta})^* (\eta+\bar{\eta})^{\beta} = (\eta^{\beta})^* \eta^{\beta}+ (\bar{\eta}^{\beta})^* \bar{\eta}^{\beta} \geq (\eta^{\beta})^* \eta^{\beta}.\]
Thus without loss of generality we can assume that $\eta:\PolQG \to \Hil$ is a proper $\tau_{\frac{i}{2}}$-real cocycle. Let $L$ be the generating functional associated to $\eta$ via Theorem \ref{cocycle->funct}. The defining formula \eqref{definingformula} (the first equality)  shows that for all $\beta \in \Irr (\QG), i,j=1, \ldots, n_{\beta}$ we have
\begin{align*} L\left(u_{ij}^{\beta} \right.&+ \left. \left(\frac{q_j(\beta)}{q_i(\beta)}\right)^{\frac{1}{2}} u_{ij}^{\beta}\right)
=  - \sum_{k=1}^{n_{\beta}} \la \eta(R(u_{ik}^\beta)^*), \eta(\tau_{\frac{i}{2}}(u_{kj})) \ra = -
 \sum_{k=1}^{n_{\beta}} \la \eta(\tau_{-\frac{i}{2}} (u_{ki})), \eta(\tau_{\frac{i}{2}}(u_{kj})) \ra
\\&= - \sum_{k=1}^{n_{\beta}} \left(\frac{q_k(\beta)}{q_i(\beta)} \right)^{\frac{1}{2}} \left(\frac{q_j(\beta)}{q_k(\beta)} \right)^{\frac{1}{2}}
 \la \eta(u_{ki}), \eta(u_{kj}) \ra  =-  \left(\frac{q_j(\beta)}{q_i(\beta)}\right)^{\frac{1}{2}} \sum_{k=1}^{n_{\beta}}  \la \eta(u_{ki}), \eta(u_{kj}) \ra,
\end{align*}
which can be rephrased as a matrix equality
\begin{equation}\label{Qbeta} L^{\beta} + (Q^{\beta})^{-\frac{1}{2}} L^{\beta} (Q^{\beta})^{\frac{1}{2}} = - (Q^{\beta})^{-\frac{1}{2}} (\eta^{\beta})^* \eta^{\beta} (Q^{\beta})^{\frac{1}{2}}.\end{equation}
In the next step we employ the averaging procedure introduced in the proof of Proposition 7.17 in \cite{DawFimSkaWhi}. Let $\mean:L^{\infty}(\br) \to \bc$ be any invariant mean (in other words, a Banach limit on $\br$). Define a new functional $\wt{L}: \PolQG \to \bc$ by the formula
\[ \wt{L}(a) = \mean \left(t \mapsto L\circ \tau_t(a)\right), \;\;\; a \in \PolQG.\]
Positivity and invariance of $\mean$, and the facts that each $\tau_t$ preserves the counit and commutes with the unitary antipode imply that  $\wt{L}$ is a generating functional which is both $R$ invariant and $\tau_t$-invariant for any $t \in \br$, and hence also $S$-invariant. It remains to see what the averaging procedure looks like in the matrix level. Fix $\beta \in \Irr_{\QG}$ and note that for each $i,j=1,\ldots,n_{\beta}$ we have $\tau_t(u_{ij}^{\beta}) = \left( \frac{q_i(\beta)}{q_j(\beta)} \right)^{it} u_{ij}^{\beta}$, so that
\[ \wt{L} (u_{ij}^{\beta}) = L(u_{ij}^{\beta}) \mean \left(t \mapsto  \left( \frac{q_i(\beta)}{q_j(\beta)} \right)^{it}\right).\]
But the action of any invariant mean on continuous periodic functions is known to yield the average over the period, so we get
\[ \wt{L} (u_{ij}^{\beta}) =\begin{cases}  L(u_{ij}^{\beta}) & \textup{ if }   q_i(\beta) = q_j(\beta) \\ 0 & \tu{ otherwise}\end{cases}.\]
The last formula can be phrased on the matrix level: let $P_1, \ldots, P_k$ denote all the spectral projections of $Q^{\beta}$, with the corresponding eigenvalues $q_1, \ldots , q_k$ (of course $k$, the individual projections and the respective eigenvalues depend on $\beta$, we do not reflect it in the notation to avoid the clutter). Then the last displayed formula can be written as
\[ \wt{L}^{\beta} = \sum_{m=1}^k P_m L^{\beta} P_m.\]
It remains to  multiply both sides of \eqref{Qbeta} by individual $P_m$ and sum the  equalities:
\[ \sum_{m=1}^k P_m  \left(L^{\beta} + (Q^{\beta})^{-\frac{1}{2}} L^{\beta} (Q^{\beta})^{\frac{1}{2}} \right) P_m = - \sum_{m=1}^k P_m \left((Q^{\beta})^{-\frac{1}{2}} (\eta^{\beta})^* \eta^{\beta} (Q^{\beta})^{\frac{1}{2}} \right) P_m\]
and as $P_m (Q^{\beta})^t = q_m^t P_m $ for any $t \in \br$ we obtain
\[ \wt{L}^{\beta} + \sum_{m=1}^k  P_m q_m^{-\frac{1}{2}} L^{\beta} q_m^{\frac{1}{2}} P_m =  - \sum_{m=1}^k P_m \left(q_m^{-\frac{1}{2}} (\eta^{\beta})^* \eta^{\beta} q_m^{\frac{1}{2}} \right) P_m\]
and further
\[ 2 \wt{L}^{\beta} = - \sum_{m=1}^k P_m (\eta^{\beta})^* \eta^{\beta}  P_m.\]
Finally let $M>0$ and choose a finite set $F \subset \Irr_\QG$ such that for any $\beta \in \Irr_{\QG} \setminus F$ we have $(\eta^{\beta})^* \eta^{\beta} \geq 2M I_{\bc^{n_{\beta}}}$ (which we can do as $\eta$ is proper). Multiplying both sides of the last inequality with $P_m$ on both sides and summing shows that also
$\sum_{m=1}^k P_m (\eta^{\beta})^* \eta^{\beta}  P_m \geq 2M \sum_{k=1}^m P_m  = 2M I_{\bc^{n_{\beta}}}$,
which together with the last displayed formula shows that $\wt{L}$ is an $S$-invariant proper generating functional. Thus Theorem 7.18 of \cite{DawFimSkaWhi} shows that $\QGam$ has the Haagerup property.
\end{proof}

\section*{Acknowledgements}


UF wants to thank the Alfried Krupp Wissenschaftskolleg in Greifswald, where he spent the summer term 2014 as research fellow, and where part of this research of carried out. AS is grateful to CNRS and the D\'epartement de math\'ematiques de Besan\c{c}on for the hospitality in autumn 2014. We all thank Michael Sch\"urmann for several useful discussions. Finally we are very grateful to the anonymous referee for a very careful reading of our manuscript and several useful remarks.

\end{document}